\documentclass{amsart}

\usepackage{macros}
\standardsettings
\colorcommentstrue

\usepackage[framemethod=TikZ]{mdframed}
\usepackage{lipsum}
\mdfdefinestyle{MyFrame}{%
 linecolor=black,
 outerlinewidth=.1pt,
 roundcorner=1pt,
 innerrightmargin=3pt,
 innerleftmargin=3pt}
\draftfalse

\begin{document}
\title[Hausdorff measure version of Gallagher's theorem]{The Hausdorff measure version of Gallagher's theorem -- closing the gap and beyond}

\authormumtaz
\authordavid

\date{}
\dedicatory{}

\begin{Abstract}

In this paper we prove an upper bound on the ``size'' of the set of multiplicatively $\psi$-approximable points in $\mathbb R^d$ for $d>1$ in terms of $f$-dimensional Hausdorff measure. This upper bound exactly complements the known lower bound, providing a ``zero-full'' law which relates the Hausdorff measure to the convergence/divergence of a certain series in both the homogeneous and inhomogeneous settings. This zero-full law resolves a question posed by Beresnevich and Velani (2015) regarding the ``$\log$ factor'' discrepancy in the convergent/divergent sum conditions of their theorem. We further prove the analogous result for the multiplicative doubly metric setup.
\end{Abstract}

\maketitle

\section{Introduction and Statement of results}
Let $\psi:\mathbb N\to [0, \infty)$ be a monotonically decreasing function such that $\psi(q)\to 0$ as $q\to\infty$. Such a function will be referred to as an \emph{approximating function}. Let $\bftheta:=(\theta_1,\ldots,\theta_d)\in\R^d$ be fixed throughout and let $\WW^d_{\psi,\bftheta}$ denote the set of $\xx=(x_1, \ldots, x_d)\in \R^d $ such that the inequality
\[
 \|qx_1 - \theta_1\|\cdots\|qx_d-\theta_d\| < \psi(q)
\]
is satisfied for infinitely many $q\in \N$. Here and throughout $\|\,\cdot\, \|$ denotes the distance from a real number to the nearest integer. The set  $\WW^d_{\psi,\bftheta}$ is a generalisation of the classical set of homogeneous multiplicatively $\psi$-approximable points, which is obtained by fixing $\bftheta=\0$. A point $\xx\in \WW^d_{\psi,\bftheta}$ will be referred to as \emph{multiplicatively $\psi$-approximable}. Multiplicative $\psi$-approximability is invariant under translation by integer vectors. Therefore throughout this note we will restrict our attention to the unit cube $\I^d = [0,1]^d$.

Multiplicative Diophantine approximation is currently an active area of research; this is in part due to the work associated with the celebrated Littlewood conjecture (1930), see \cite{BadziahinVelani, EKL}. The Littlewood conjecture states that $\WW^2_{q\mapsto\epsilon q^{-1},\0}=\R^2$ for all $\epsilon>0$. 

A natural question is to determine the ``size'' of the set $\WW^d_{\psi,\bftheta}\cap\I^d$ in terms of Lebesgue measure, Hausdorff measure, and Hausdorff dimension, which are natural tools for this purpose. Gallagher \cite{Gallagher2} proved the following Lebesgue measure statement for the set of homogeneous multiplicatively $\psi$-approximable points $\WW^d_{\psi,\0}$. Here and throughout, for any $s>0$, $\HH^s (X)$ stands for the $s$-dimensional Hausdorff measure of the set $X$. In the case $s=d$, the $s$-dimensional Hausdorff measure is comparable with the $d$-dimensional Lebesgue measure. For completeness, definitions of Hausdorff measure and dimension are provided below; for further details please refer to \cite{Falconer_book}.

\begin{theorem}[Gallagher 1962]\label{thm1}
Let $\psi$ be any approximating function. Then,
\[
\HH^d(\WW^d_{\psi,\0}\cap\I^d)=\begin {cases}
0 \ & {\rm if } \quad \sum\limits_{q=1}^{\infty}\psi(q)\log^{d-1}(q)<\infty. \\[3ex]
1 \ & {\rm if } \quad \sum\limits_{q=1}^{\infty}\psi(q) \log^{d-1}(q)=\infty.
\end {cases}
\]
where $\HH^d$ is normalized so that $\HH^d(\I^d) = 1$ (we will use this convention throughout the paper).
\end{theorem}

The monotonicity assumption on the approximating function is only needed for the divergence case, and there only when $d=1$. For higher dimensions it can be removed, see \cite{BHV3} and references therein for further details.

The analogue of Gallagher's result for the inhomogeneous multiplicative set $\WW^d_{\psi,\bftheta}$ is surprisingly incomplete. A straightforward application of the first Borel--Cantelli lemma implies that, for any approximating function $\psi$,
 \begin{equation}\label{ibclem}
\HH^d(\WW^d_{\psi,\bftheta}\cap\I^d)=0\quad {\rm if} \quad \sum\limits_{q=1}^{\infty}\psi(q)\log^{d-1}(q)<\infty.
\end{equation}
However, the divergence counterpart of \eqref{ibclem} is unknown. We attribute the following conjecture to Beresnevich, Haynes, and Velani \cite[Conjecture 2.1]{BHV4} where the statement appeared for the first time for $d=2$.

\begin{conjecture*}[Beresnevich--Haynes--Velani]
For any approximating function $\psi$,
$$\HH^d(\WW^d_{\psi,\bftheta}\cap\I^d)=1 \quad {\rm if} \quad \sum\limits_{q=1}^{\infty}\psi(q)\log^{d-1}(q)=\infty.$$
\end{conjecture*}
The only known result towards this conjecture was very recently established in \cite{BHV4}. It only covers the $d=2$ case and only for the situation when one of the coordinates of the inhomogeneous parameter is zero, i.e. when $\theta_1=0$ or $\theta_2=0$.

Naturally as a next step one would like to obtain the Hausdorff measure version of Gallagher's theorem or even of the conjectural inhomogeneous analogue. But it would be surprising to have any result of this kind when the Lebesgue inhomogeneous theory is so incomplete. Surprisingly, very recently--in 2015, Bersenevich--Velani \cite{BeresnevichVelani3} obtained a partial Hausdorff measure version of the inhomogeneous Gallagher's theorem. 

\begin{theorem}[Beresnevich--Velani 2015]\label{thm2}
 Let $\psi$ be any approximating function and $s\in(d-1, d)$ for $d\geq 2$. Then,
\begin{equation}
\label{2}
\HH^s(\WW^d_{\psi,\bftheta}\cap\I^d)=
\begin {cases}
 0 \ & {\rm if } \quad \sum\limits_{q=1}^{\infty}q^{d-s}\psi^{s-d+1}(q)\log^{d-2}(q)<\infty. \\[3ex]
 \infty \ & {\rm if } \quad \sum\limits_{q=1}^{\infty}q^{d-s}\psi^{s-d+1}(q) =\infty.
\end {cases}
\end{equation}
\end{theorem}

It is important to note that the above result used the monotonicity assumption on the approximating function for both the convergence and divergence cases. The assumption that $s\in(d-1,d)$ is natural as it can be readily verified that
\begin{itemize}
\item when $s\leq d-1$, $\HH^s(\WW^d_{\psi,\bftheta}\cap\I^d)=\infty$ irrespective of the approximating function $\psi$. 
\item when $s>d, \HH^s(\WW^d_{\psi,\bftheta}\cap\I^d)=0$ irrespective of the approximating function $\psi$.
\item when $s=d$, the Beresnevich--Haynes--Velani conjecture is contradictory to \eqref{2} and thus we do not expect \eqref{2} to hold when $s=d$.
\end{itemize}

The convergent sum condition in Theorem \ref{thm2} carries an extra log factor when compared to the divergent sum condition. This means that there are some approximating functions $\psi$ for which neither condition holds, such as the functions
\[
\psi(q) = \left(q^{d-s+1} \log^\alpha(q)\right)^{-1/(s-d+1)} \;\;\;\;\; 1 < \alpha \leq d-1,
\]
and thus Theorem \ref{thm2} is insufficient to compute the Hausdorff measure of $\WW^d_{\psi,\bftheta}\cap\I^d$ for these functions. Beresnevich and Velani describe the situation as follows (see \cite[Remark 1.2]{BeresnevichVelani3})\footnote{Note that in \cite{BeresnevichVelani3}, $n$ was used instead of $d$.}:
\medskip

``{\it Thus there is a discrepancy in the above `$s$-volume' sum conditions for convergence and divergence when $n>2$. In view of this it remains an interesting open problem to determine the necessary and sufficient condition for $\HH^s(\WW^n_{\psi,\bftheta}\cap\I^n)$ to be zero or infinite in higher dimensions.}''

\medskip

Theorem \ref{thm3}, given below, proves that the log factor in Theorem \ref{thm2} is redundant for any $d$. It is in fact more powerful as the convergence case is free from any monotonicity assumption on the approximating function for any $d$. The monotonicity is only needed for the divergence case. The results are strengthened further by considering the Hausdorff measure of $\WW^d_{\psi,\bftheta}\cap\I^d$ with respect to a general \emph{dimension function} $f$. A dimension function is an increasing continuous function $f:(0,\infty)\to (0,\infty)$ such that $f(r)\to 0$ as $r\to 0$. We need to impose the following technical restrictions on $f$:
\begin{itemize}
\item[(I)] there exist $s\in (d-1,d)$ and $C > 0$ such that
\begin{equation}
\label{fbound}
f(y) \leq C(y/x)^s f(x) \;\;\;\;\forall  \ 0 < x < y.
\end{equation}
\item[(II)] the map $x\mapsto x^{-d+1} f(x)$ is monotonically increasing.
\end{itemize}
Obviously, both (I) and (II) are satisfied for the dimension function $f(x) = x^s$ whenever $s\in (d-1,d)$. They are also satisfied for more general functions such as $f(x) = x^s \log^\alpha(x) \log^\beta\log(x)$ (where $\alpha,\beta\in\R$), etc.

\begin{theorem}\label{thm3}
 Fix $\psi:\N\to\Rplus$ and a dimension function $f$ satisfying \text{(I)} and \text{(II)}. Then
 $$\HH^f(\WW^d_{\psi,\bftheta}\cap\I^d)=\begin {cases}
 0 \ & {\rm if } \quad \sum\limits_{q=1}^{\infty} q^d \psi^{-d+1}(q) f\left(\frac{\psi(q)}{q}\right)<\infty. \\[3ex]
 \infty \ & {\rm if } \quad\sum\limits_{q=1}^{\infty} q^d \psi^{-d+1}(q) f\left(\frac{\psi(q)}{q}\right) =\infty \text{ and $\psi$ is monotonic}.
 \end {cases}$$
\end{theorem}

\begin{corollary}\label{cor3}
 Fix $\psi:\N\to\Rplus$ and $s\in (d-1,d)$. Then
 $$\HH^s(\WW^d_{\psi,\bftheta}\cap\I^d)=\begin {cases}
 0 \ & {\rm if } \quad \sum\limits_{q=1}^{\infty} q^{d-s} \psi^{s-d+1}(q) <\infty. \\[3ex]
 \infty \ & {\rm if } \quad\sum\limits_{q=1}^{\infty} q^{d-s} \psi^{s-d+1}(q) =\infty \text{ and $\psi$ is monotonic}.
 \end {cases}$$
\end{corollary}

\textbf{Note:} The methods needed to prove the divergence case are exactly the same as in \cite{BeresnevichVelani3} but since these are very short we include them for completeness. \\

{\bf Acknowledgements.}  The first-named author was supported by the Endeavour Fellowship--Department of Education and Training, Australia.  The second-named author was supported by the EPSRC Programme Grant EP/J018260/1. Part of this work was carried out when the second-named author visited the University of Newcastle. We are thankful to the  University of Newcastle and Australian Mathematical Sciences Institute (AMSI) for travel support. We would like to thank the anonymous referee for useful comments.

\section{Proofs}

For completeness we give below a very brief introduction to Hausdorff measures and dimension. For further details see \cite{Falconer_book}.

Let
$S\subset \R^d$.
 Then for any $\rho>0$, any finite or countable collection~$\{B_i\}$ of subsets of $\R^d$ with diameters $\mathrm{diam} (B_i)\le \rho$ such that
$S\subset \bigcup_i B_i$ is called a \emph{$\rho$-cover} of $S$.
Let
\[
\HH_\rho^f(S)=\inf \sum_i f\left(\diam (B_i)\right),
\]
where the infimum is taken over all possible $\rho$-covers $\{B_i\}$ of $S$. The \textit{Hausdorff $f$-measure of $S$} is defined to be
\[
\HH^f(S)=\lim_{\rho\to 0}\HH_\rho^f(S).
\]
The map $\HH^f:\P(\R^d)\to [0,\infty]$ is a Borel measure. In the case that $f(r)=r^s \;\; (s\geq 0)$, the measure $\HH^f$ is denoted $\HH^s$ and is called \emph{$s$-dimensional Hausdorff measure}. For any set $S \subset \R^d$ one can easily verify that there exists a unique critical value of $s$ at which the function $s\mapsto\HH^s(S)$ ``jumps'' from infinity to zero. The value taken by $s$ at this discontinuity is referred to as the \textit{Hausdorff dimension} of $S$ is denoted by $\dim_{\HH} S $; i.e.
\[
\dim_\HH S :=\inf\{s\geq 0\;:\; \HH^s(S)=0\}.
\]
The countable collection $\{B_i\}$ is called a \emph{fine cover} of $S$ if for every $\rho>0$ it contains a subcollection that is a $\rho$-cover of $S$.

\medskip

There are many benefits of a characterisation of $\WW^d_{\psi,\bftheta}$ using Hausdorff measure. As stated above, one such benefit is that such a characterisation implies a formula for the Hausdorff dimension of $\WW^d_{\psi,\bftheta}$ (see \cite[\sectionsymbol5 and \sectionsymbol12.7]{BDV}). In particular, let $\tau$ be the lower order at infinity of $1/\psi$, that is,
\[\tau:=\liminf_{q\to\infty}\frac{\log(1/\psi(q))}{\log q}.\]
Then Theorem \ref{thm3} (and in fact the weaker Theorem \ref{thm2}) implies that for any approximating function $\psi$ with lower order at infinity $\tau$, we have
\[
s_0:=\dim_\HH \WW^d_{\psi,\bftheta}=\begin {cases}
 d \ & {\rm if } \quad \tau \leq 1. \\[3ex]
 d+\frac{1-\tau}{1+\tau} \ & {\rm if } \quad \tau>1.
 \end {cases}
\]
Note that the limiting dimension $\lim_{\tau\to\infty}\dim_\HH \WW^d_{\psi,\bftheta} = d-1$ is never zero unless $d=1$. In fact, Theorem \ref{thm2} reveals much more than just the Hausdorff dimension: it can readily be verified that \[\HH^{s_0}(\WW^d_{\psi,\bftheta}\cap\I^d)=\infty.\]

\subsection{Proof of Theorem \ref{thm3}: the convergence case} We first state the Hausdorff measure version of the famous Borel--Cantelli lemma (see \cite[Lemma 3.10]{BernikDodson}) which will allow us to estimate the Hausdorff measure of certain sets via calculating the Hausdorff $f$-sum of a fine cover.

\begin{lemma}[Hausdorff--Cantelli]\label{bclem}
Let $\{B_i\}\subset\R^d$ be a fine cover of a set $S$ and let $f$ be a dimension function such that
\begin{equation}
\label{fdimcost}
\sum_i f\left(\diam(B_i)\right) \, < \, \infty.
\end{equation}
Then $$\HH^f(S)=0.$$
\end{lemma}
In what follows we will call the series \eqref{fdimcost} the \emph{$f$-dimensional cost} of the collection $\{B_i\}$. Note that if $\{B_i\}$ is only a cover and not a fine cover, then there is no necessary relation between the $f$-dimensional cost and $f$-dimensional Hausdorff measure.

The next lemma will be key in forming a fine cover of $\WW^d_{\psi,\bftheta}\cap \I^d$ and thus estimating its $f$-Hausdorff measure.

\begin{lemma}\label{lem2}
Fix $s \in (d-1,d)$ and for each $0 < r \leq 1$ let
\[
M(r) = \left\{\xx\in\R^d : |x_i| \leq 1, \; \prod_i |x_i| \leq r\right\}.
\]
Then $M(r)$ can be covered by a collection of $d$-dimensional hypercubes $\{B_i\}$ satisfying
\begin{align} \label{diambound}
\diam(B_i) &\geq r\\ \label{costbound}
\sum_i\diam^s(B_i) &\lessless r^{s-d+1}.
\end{align}
\end{lemma}

The notation $\lessless$ is used to indicate an inequality with an unspecified positive multiplicative constant. If $a\lessless b$ and $a\gtrgtr b$ we write $a\asymp b$, and say that the quantities $a$ and $b$ are \emph{comparable}.

\begin{proof}[Proof of Theorem \ref{thm3}: the convergence case modulo Lemma \ref{lem2}]
We will assume that
\[
\sum\limits_{q=1}^{\infty} q^d \psi^{-d+1}(q) f\left(\frac{\psi(q)}{q}\right) <\infty.
\]
For each $q\in\N$ let
\[
Z_q = \{\pp\in\Z^d : -1 \leq p_i + \theta_i < q + 1 \text{ for all $i$}\}.
\]
It is easily verified that
\[
\WW^d_{\psi,\bftheta}\cap \I^d= \bigcap_{N=1}^\infty\bigcup_{q=N}^\infty\bigcup_{\pp\in Z_q} \frac{1}{q}\big(\pp + \bftheta + M(\psi(q))\big).
\]
Thus the collection
\[
\left\{\frac{1}{q}\big(\pp + \bftheta + M(\psi(q))\big) : q\in \N ,\; \pp\in Z_q\right\}
\]
is a fine cover of $\WW^d_{\psi,\bftheta}\cap \I^d$. Now for each $q$, let $\{B_{q,i} : i = 1,\ldots,N_q\}$ be a covering of $M(\psi(q))$ by $d$-dimensional hypercubes as in Lemma \ref{lem2}. Then the collection
\[
\left\{\frac{1}{q}\big(\pp + \bftheta + B_{q,i}\big) : q\in \N ,\; \pp\in Z_q,\; i = 1,\ldots,N_q\right\}
\]
is also a fine cover of $\WW^d_{\psi,\bftheta}\cap \I^d$. The $f$-dimensional cost of this cover is
\begin{align*}
&\sum_{q\in\N} \sum_{\pp\in Z_q} \sum_{i = 1}^{N_q} f\left(\diam\left(\frac{1}{q}\big(\pp + \bftheta + B_{q,i}\big)\right)\right)\\
&\lessless \sum_{q\in\N} \sum_{\pp\in Z_q} \sum_{i = 1}^{N_q} \left(\frac{\diam\big(\frac{1}{q}\big(\pp + \bftheta + B_{q,i}\big)\big)}{\psi(q)/q}\right)^s f\left(\frac{\psi(q)}{q}\right) \by{\eqref{fbound} and \eqref{diambound}}\\
&= \sum_{q\in\N} (q+2)^d \sum_{i = 1}^{N_q} \left(\frac{\diam(B_{q,i})}{\psi(q)}\right)^s f\left(\frac{\psi(q)}{q}\right)\\
&\lessless \sum_{q\in\N} q^d \left(\frac{\psi^{s-d+1}(q)}{\psi^s(q)}\right) f\left(\frac{\psi(q)}{q}\right) \by{\eqref{costbound}}\\
&= \sum\limits_{q=1}^{\infty} q^d \psi^{-d+1}(q) f\left(\frac{\psi(q)}{q}\right).
\end{align*}
By assumption, this series converges and thus by the Hausdorff--Cantelli lemma (Lemma \ref{bclem}), we have $\HH^f(\WW^d_{\psi,\bftheta}\cap \I^d) = 0$.
\end{proof}

Finally to complete the proof of the convergence case of Theorem \ref{thm3} the only thing remains to show is to prove Lemma \ref{lem2}.

\begin{proof}[Proof of Lemma \ref{lem2}]
Without loss of generality assume that $r \leq 2^{-d}$. Let $N\geq d$ be the largest integer such that $r \leq 2^{-N}$. Let
\[
S = \{\kk = (k_1,\ldots,k_d) \in \Z^d : k_1,\ldots,k_d \geq 0, \; k_1 + \ldots + k_d = N - d\}
\]
and for each $\kk\in S$ let
\[
B(\kk) = \prod_{i = 1}^d [-2^{-k_i},2^{-k_i}].
\]
Fix $\xx\in M(r)$, and for each $i$, let $k_i$ be the largest integer such that $|x_i| \leq 2^{-k_i}$. Since $|x_i| \leq 1$, $k_i \geq 0$, and since $\prod_i |x_i| \leq r$, $\sum_i (k_i + 1) \geq N$. Thus there exists $\kk' \in S$ such that $k'_i \leq k_i$ for all $i$, and thus
\[
M(r) \subset \bigcup_{\kk\in S} B(\kk).
\]
Fix $\kk\in S$ such that $\max(k_1,\ldots,k_d) = k_1$. Then for each $i$, the interval $[-2^{-k_i},2^{-k_i}]$ can be written as the essentially disjoint union of $2^{k_1-k_i+1}$ intervals of length $2^{-k_1}$. Thus, $B(\kk)$ can be written as the essentially disjoint union of
\[
\prod_{i = 1}^d 2^{k_1-k_i+1} = 2^{d k_1 - \sum_i k_i + d} = 2^{d k_1 - N}
\]
hypercubes of side length $2^{-k_1}$. In general, for all $\kk\in S$, $B(\kk)$ can be written as the essentially disjoint union of $2^{d k_{\max} - N}$ hypercubes of side length $2^{-k_{\max}}$, where $k_{\max} = \max(k_1,\ldots,k_d)$. Denote this collection of hypercubes by $\{B_{\kk,i} : i = 1,\ldots,2^{d k_{\max} - N}\}$. Then
\begin{equation}
\label{coveringMr}
\{B_{\kk,i} : \kk\in S, \; i = 1,\ldots,2^{d k_{\max} - N}\}
\end{equation}
is a covering of $M(r)$ by hypercubes of diameter $\geq r$.

To complete the proof, we must show that the $s$-dimensional cost of the covering \eqref{coveringMr} is $\lessless r^{s-d+1}$. And indeed,
\begin{align*}
\sum_{\kk\in S} \sum_{i = 1}^{2^{d k_{\max} - N}} \diam^s(B_{\kk,i})
&\asymp \sum_{\kk\in S} 2^{d k_{\max} - N} (2^{-k_{\max}})^s\\
&\asymp \sum_{k_1 = 0}^{N - d} \sum_{\substack{0 \leq k_2,\ldots,k_d \leq k_1 \\ \sum_i k_i = N - d}} 2^{(d - s) k_1 - N} \by{symmetry}\\
&\leq \sum_{k=0}^{N - d} \#\Big\{k_2,\ldots,k_d \geq 0 : \sum_{i \geq 2} k_i = N - d - k\Big\} 2^{(d - s) k - N} \noreason\\
&\asymp \sum_{k = 0}^{N - d} (N - d - k)^{d - 2} 2^{(d - s) k - N}\\
&= \sum_{\ell = 0}^{N - d} \ell^{d - 2} 2^{(d - s) (N - d - \ell) - N} \note{letting $\ell = N - d - k$}\\
&\leq 2^{(d - s) N - N} \sum_{\ell = 0}^\infty \ell^{d - 2} 2^{-(d - s)(d + \ell)} \note{converges since $s < d$}\\
&\asymp 2^{(d - s) N - N}=\left(2^{-N}\right)^{s-d+1} \asymp r^{s - d + 1}. \ \ 
\end{align*}
This completes the proof of Lemma \ref{lem2}.

\end{proof}

\begin{remark}
We can see the factor which corresponds to the ``log discrepancy'' appear in the calculation on the fourth line, namely the factor $(N - d - k)^{d - 2}$. The trick to getting rid of it is the change of variables $\ell = N - d - k$, which makes the factor into part of a convergent series which is independent of $N$. Heuristically, the reason this trick works is that the asymptotic $(N - d - k)^{d - 2} \asymp |\log(r)|^{d - 2}$ is only valid for small values of $k$, and for values of $k$ closer to $N-d$, the asymptotic $(N - d - k)^{d - 2} \asymp 1$ is more accurate. Since the second factor in the sum is exponentially increasing with $k$, the largest possible values of $k$ contribute the most to the sum.
\end{remark}

\subsection{Proof of Theorem \ref{thm3}: the divergence case}

The proof will proceed using the following ``slicing lemma'' for Hausdorff measures, see \cite[Proposition 7.9]{Falconer_book} or \cite[Lemma 4]{BeresnevichVelani}.

\begin{lemma}[Slicing Lemma]
\label{lemmaslicing}
Fix $k,l\in\N$ with $l < k$, and let $g$ be a dimension function and let $f(r) = r^\ell g(r)$ (note that $f$ is necessarily a dimension function). Let $A$ be a Borel subset of $\R^k$ and suppose that the set
\[
\big\{x \in \R^\ell: \HH^g(\{y\in \R^{k-\ell} : (x,y)\in A\}) = \infty\big\}
\]
has positive $\HH^\ell$-measure. Then $\HH^f(A) = \infty$.
\end{lemma}

Let $\psi$ be a decreasing function and let $\theta_1\in\R$ be fixed. Consider the set of inhomogeneous $\psi$-approximable real numbers
\[
\WW^1_{\psi,\theta_1} = \left\{x_1\in \R : \; \|qx_1 - \theta_1\| < \psi(q)\quad {\rm for \ i. m. } \ q\in \N\right\}.
\]

 In \cite{Bugeaud5}, Bugeaud proved the following zero--infinity law for $\WW^1_{\psi,\theta_1} $.
 
 \begin{theorem}[Bugeaud 2004]\label{thm5}
 Let $\psi$ be an approximating function and let $g$ be a dimension function. Then
 $$\HH^g(\WW^1_{\psi,\theta_1} )=\begin {cases}
 0 \ & {\rm if } \quad \sum\limits_{r=1}^{\infty}rg\left(\frac{\psi(r)}{r}\right)<\infty. \\[3ex]
 \infty \ & {\rm if } \quad \sum\limits_{r=1}^{\infty}rg\left(\frac{\psi(r)}{r}\right)=\infty.
 \end {cases}$$
\end{theorem}

 As in \cite{BeresnevichVelani3}, we notice that $\WW^1_{\psi,\theta_1} \times \R^{d-1}\subset \WW^d_{\psi,\bftheta}$ for any $\bftheta \in \R^d$. Thus, $$\HH^f\left(\WW^d_{\psi,\bftheta}\cap \I^d\right)\geq \HH^f\left((\WW^1_{\psi,\theta_1} \cap\I) \times \I^{d-1}\right)$$
for any dimension function $f$. Now if $f$ satisfies (II), then the function $g$ defined by the equation $f(r)=r^{d-1}g(r)$ is a dimension function, and the divergence case of Theorem \ref{thm5} implies that $\HH^g(\WW^1_{\psi,\theta_1} \cap \I)=\infty$. Now using the Slicing Lemma, we obtain that
 
 $$\HH^f\left(\WW^d_{\psi,\bftheta}\cap \I^d\right)=\infty \quad {\rm if} \quad \sum\limits_{r=1}^{\infty} r^d \psi^{-d+1}(r) f\left(\frac{\psi(r)}{r}\right)= \sum\limits_{r=1}^{\infty}rg\left(\frac{\psi(r)}{r}\right)=\infty.$$

\section{Remarks on the doubly metric case}

Recall that throughout this paper, so far, we have fixed the inhomogeneous parameter $\bftheta\in \I^d$. If $\bftheta$ is not fixed and is allowed to vary then the resulting setup is slightly different than what has been discussed above. Such a setup is commonly known as \emph{doubly metric}. Let $\psi$ be an approximating function. Consider the doubly metric set
\[
\WW^d_\psi = \{(\bftheta,\xx) \in \I^{2d}: \xx\in \WW^d_{\psi,\bftheta}\}.
\]
With the help of the first and second Borel-Cantelli lemmas, it is easy to show that for any function $\psi:\N\to \CO 0\infty$,
 \begin{equation}\label{dibclem}
\HH^{2d}(\WW^d_\psi\cap\I^{2d}) =
\begin{cases}
0 & {\rm if} \quad \sum\limits_{q=1}^{\infty}\psi(q)\log^{d-1}(q)<\infty.\\
1 & {\rm if} \quad \sum\limits_{q=1}^{\infty}\psi(q)\log^{d-1}(q) = \infty.
\end{cases}
\end{equation}
This theorem can be proven in the same manner as in  \cite[Chapter VII, \61--4]{Cassels}. It can be readily verified that the proof of the divergence case follows upon replacing the function $\delta_q$ defined on p.123  in \cite{Cassels} with the appropriate multiplicative version and then redoing the calculations. The reason that the doubly metric case is easier to deal with than the singly metric case is the fact that the inhomogeneous variable offers an extra degree of freedom which, in this instance, makes calculations a lot easier than the situation when the inhomogeneous parameter is fixed. 
However, to the best of our knowledge nothing is known regarding the measure-theoretic properties of $\WW^d_\psi$ with respect to Hausdorff measures. 

\medskip

Now we turn our attention to the Hausdorff measure of the set $\WW^d_\psi$. Without much effort we prove the following doubly metric analogue of Theorem \ref{thm3}.

\begin{theorem}
\label{thm6}
Fix $\psi:\N\to\R$ and a dimension function $f$ satisfying \text{(I)} and \text{(II)} of Theorem \ref{thm3}. Let $F(x) = x^d f(x)$. Then
\[
\HH^F(\WW^d_\psi\cap\I^{2d})=\begin {cases}
0 \ & {\rm if } \quad \sum\limits_{q=1}^{\infty} q^d \psi^{-d+1}(q) f\left(\frac{\psi(q)}{q}\right)<\infty. \\[3ex]
\infty \ & {\rm if } \quad \sum\limits_{q=1}^{\infty} q^d \psi^{-d+1}(q) f\left(\frac{\psi(q)}{q}\right) =\infty \text{ and $\psi$ is monotonic}.
\end {cases}
\]
\end{theorem}

\begin{remark*}
It is reasonable to ask whether the monotonicity hypothesis is necessary in this theorem, or whether it can be removed as in \eqref{dibclem}. Though the Slicing Lemma appears unlikely to yield such an improvement, perhaps it could be achieved via other doubly metric techniques.
\end{remark*}

Theorem \ref{thm6} implies that for any approximating function $\psi$ with lower order at infinity $\tau$, we have

$$\dim_\HH \WW^d_\psi=\begin {cases}
 2d \ & {\rm if } \quad \tau \leq 1. \\[3ex]
 2d+\frac{1-\tau}{1+\tau} \ & {\rm if } \quad \tau>1.
 \end {cases}$$
 
Note that the limiting dimension $\lim_{\tau\to\infty}\dim_\HH \WW^d_\psi = 2d-1$ is never zero. 

\begin{proof}[Proof of Theorem \ref{thm6}]
The divergence case follows immediately upon combining Theorem \ref{thm3} with the Slicing Lemma just as in the proof of Theorem \ref{thm3}. To prove the convergence case, we use an argument similar to the proof of Theorem \ref{thm3}. As before we assume that
\begin{equation}
\sum\limits_{q=1}^{\infty} q^d \psi^{-d+1}(q) f\left(\frac{\psi(q)}{q}\right)<\infty,
\end{equation}
and we assume without loss of generality that $f(2r) \asymp f(r)$ for all $r$.\Footnote{If not, then let $g(r) = \inf_{n\in\N} n^{2d} f(r/n)$. Since every subset of $\R^{2d}$ of diameter $r$ can be covered by $\asymp n^{2d}$ sets of diameter $r/n$, we have $\HH^g(S) \asymp \HH^f(S)$ for all $S \subset \R^{2d}$ and thus it suffices to prove the theorem for the function $g$.} It is easily verified that
\begin{equation}
\WW^d_\psi\cap \I^{2d}= \bigcap_{N=1}^\infty\bigcup_{q=N}^\infty\bigcup_{\pp\in Z_q} \left\{(\bftheta,\xx) : \xx\in \frac{1}{q}\big(\pp + \bftheta + M(\psi(q))\big)\right\},
\end{equation}
where $Z_q$ is as before. Thus the collection
\[
\left\{\left\{(\bftheta,\xx) : \xx\in \frac{1}{q}\big(\pp + \bftheta + M(\psi(q))\big)\right\} : q\in \N ,\; \pp\in Z_q\right\}
\]
is a fine cover of $\WW^d_\psi\cap \I^{2d}$. Now for each $q$, let $\{B_{q,i} : i = 1,\ldots,M_q\}$ be a covering of $M(\psi(q))$ by $d$-dimensional hypercubes as in Lemma \ref{lem2}. For each $r > 0$, let $\{A_j(r) : j = 1,\ldots,N(r)\}$ be a covering of $\I^d$ by $N(r) \asymp r^{-d}$ $d$-dimensional hypercubes of diameter $r$. For shorthand, write $N_{q,i} = N(q^{-1} \diam(B_{q,i}))$ and $A_{q,i,j} = A_j(q^{-1}\diam(B_{q,i}))$. Then the collection
\[
\left\{A_{q,i,j} \times \frac{1}{q}\big(\pp + A_{q,i,j} + B_{q,i}\big) : q\in \N ,\; \pp\in Z_q,\; i = 1,\ldots,M_q, \; j = 1,\ldots,N_{q,i}\right\}
\]
is also a fine cover of $\WW^d_\psi\cap \I^{2d}$. Here $A+B$ denotes the Minkowski sum of two sets $A$ and $B$. The $F$-dimensional cost of this cover is\Footnote{For convenience the following calculations will be done with respect to the max norm on $\R^{2d}$.}
\begin{align*}
&\sum_{q\in\N} \sum_{\pp\in Z_q} \sum_{i = 1}^{M_q} \sum_{j=1}^{N_{q,i}} F\left(\diam\left(A_{q,i,j}\times\frac{1}{q}\big(\pp + A_{q,i,j} + B_{q,i}\big)\right)\right) \noreason\\
&= \sum_{q\in\N} \sum_{i = 1}^{M_q} (q+2)^d N_{q,i} F\left(\max\left(q^{-1} \diam(B_{q,i}),\frac{1}{q}\Big(q^{-1} \diam(B_{q,i}) + \diam(B_{q,i})\Big)\right)\right) \noreason\\
&\asymp \sum_{q\in\N} q^d \sum_{i = 1}^{M_q} (q^{-1}\diam(B_{q,i}))^{-d} F\left(\frac{1}{q}\diam(B_{q,i})\right) \since{$f(2r) \asymp f(r)$}\\
&= \sum_{q\in\N} q^d \sum_{i = 1}^{M_q} f\left(\frac{1}{q} \diam(B_{q,i})\right)\\
&\lessless \sum_{q\in\N} q^d \sum_{i = 1}^{M_q} \left(\frac{\diam(B_{q,i})}{\psi(q)}\right)^s f\left(\frac{\psi(q)}{q}\right). \by{\eqref{fbound} and \eqref{diambound}}
\end{align*}
The calculation in the proof of Theorem \ref{thm3} shows that this series converges and thus by the Hausdorff--Cantelli lemma, we have $\HH^F(\WW^d_\psi\cap \I^{2d}) = 0$.
\end{proof}

\section{Final comments and open problems}

Let $m\geq 1$ and $d\geq 1$ be integers. Let $\Psi:\Z^m\to \CO 0\infty$ be a  ``multivariable approximating function'', and fix $\bftheta\in\I^d$. Consider the set
\[
\WW^{md}_{\Psi,\bftheta}:=\left\{X:=(\mathbf x^{(1)},\ldots, \mathbf x^{(d)})\in \I^{md}: \prod_{i=1}^d\|\mathbf q\cdot\mathbf x^{(i)}-\theta_i\|<\Psi(\mathbf q)\quad {\rm for \ infinitely \ many} \ \mathbf q \in \Z^m\setminus\{\bf 0\}\right\}.
\]
Naturally one would like to establish a coherent metric theory for this set. To the best of our knowledge nothing has been proven for this set except a zero--one law for the homogeneous set $\WW^{md}_{\Psi,\bf0}$ which was proved in \cite{BHV3}. That is, 
\[
\HH^{md}(\WW^{md}_{\Psi,\bf0})\in\{0,1\}.
\]
The main tool used in establishing this result was the ``cross-fibering principle''  which made it possible to lift Cassel's and Gallagher's zero--one laws to the higher-dimensional multiplicative setup.
Beyond this result nothing is known. It can be verified that the set $\WW^{md}_{\Psi,\bftheta}$ has Lebesgue measure zero if the sum
\begin{equation}\label{eqcon}
\sum_{\qq\in  \Z^m\setminus\{\bf 0\}}\Psi^m(\mathbf q)\log^{d-1}(|\mathbf q|)
\end{equation}
converges. Here $|\qq|$ denotes the sup norm of $\qq$.  However, if \eqref{eqcon} diverges, then the situation is less clear. As in Theorem \ref{thm3}, it is reasonable to expect some monotonicity assumption for the divergence case. We therefore make the following conjecture:

\begin{conjecture}
Let $\Psi:\Z^m\to \CO 0\infty$, and suppose that $\Psi(\qq)=\psi(|\qq|)$ for some monotonic function $\psi$. Then
\[\HH^{md}(\WW^{md}_{\Psi,\bftheta})=1\quad {\rm if}\quad \eqref{eqcon} \ {\rm diverges}.\]
\end{conjecture}


\medskip

Using methods similar to those used in this paper, one can show that if the series
\begin{equation}
\label{Psiconverges}
\sum_{\qq\in\Z^m} |\qq|^{md} \Psi^{-md+1}(\qq) f\left(\frac{\Psi(\qq)}{|\qq|}\right)
\end{equation}
converges, then
\[
\HH^f(\WW^{md}_{\Psi,\bftheta}\cap \I^{md}) = 0,
\]
where $f$ is a dimension function satisfying \eqref{fbound} for some $s\in (md-1,md)$ and $C > 0$. This leads us to the following complementary problem.
\begin{conjecture}
Let $\Psi(\qq) = \psi(|\qq|)$, where $\psi:\N\to \CO 0\infty$ is a monotonically decreasing function. Then if \eqref{Psiconverges} diverges and $x\mapsto x^{-md+1} f(x)$ is monotonically increasing, then $\HH^f(\WW^{md}_{\Psi,\bftheta}\cap \I^{md}) = \infty$.
\end{conjecture}

Note that this conjecture does not follow from the Slicing Lemma, since although $\WW^{md}_{\Psi,\bftheta}$ contains certain slices of the form $\WW^d_{\psi,\bftheta}\times \I^{(m-1)d}$, the corresponding function $\psi:\N\to \CO 0\infty$ is necessarily different from $\Psi$ (since they have different domains) and it may happen that the series \eqref{Psiconverges} diverges whereas the analogous series for $\psi$ converges.

We also conjecture that in the doubly metric case, the monotonicity hypothesis on $\Psi$ can be removed:

\begin{conjecture}
Let $\Psi:\Z^m\to \CO 0\infty$ be any function, and let $f$ be a dimension function such that \eqref{Psiconverges} diverges and $x\mapsto x^{-md+1} f(x)$ is monotonically increasing. Let $F(x) = x^{md} f(x)$ and
\[
\WW^{md}_\Psi = \{(\bftheta,X) : X \in \WW^{md}_{\Psi,\bftheta}\}.
\]
Then $\HH^F(\WW^{md}_\Psi\cap \I^{2md}) = \infty$.
\end{conjecture}
Again, the convergence case analogue can be proven using the techniques of this paper.


\begin{thebibliography}{10}

\bibitem{BadziahinVelani}
Dzmitry Badziahin and Sanju Velani, \emph{Multiplicatively badly approximable
  numbers and generalised {C}antor sets}, Adv. Math. \textbf{228} (2011),
  no.~5, 2766--2796. \MR{2838058}

\bibitem{BDV}
Victor Beresnevich, Detta Dickinson, and Sanju Velani, \emph{Measure theoretic
  laws for lim sup sets}, Mem. Amer. Math. Soc. \textbf{179} (2006), no. 846,
  x+91 pp.

\bibitem{BHV3}
Victor Beresnevich, Alan Haynes, and Sanju Velani, \emph{Multiplicative
  zero-one laws and metric number theory}, Acta Arith. \textbf{160} (2013),
  no.~2, 101--114. \MR{3105329}

\bibitem{BHV4}
\bysame, \emph{Sums of reciprocals of fractional parts and multiplicative
  {D}iophantine approximation}, \url{https://arxiv.org/abs/1511.06862},
  preprint 2015.

\bibitem{BeresnevichVelani}
Victor Beresnevich and Sanju Velani, \emph{Schmidt's theorem, {H}ausdorff
  measures, and slicing}, Int. Math. Res. Not. (2006), Art. ID 48794, 24.

\bibitem{BeresnevichVelani3}
Victor Beresnevich and Sanju Velani, \emph{A note on three problems in metric
  {D}iophantine approximation}, Recent trends in ergodic theory and dynamical
  systems, Contemp. Math., vol. 631, Amer. Math. Soc., Providence, RI, 2015,
  pp.~211--229. \MR{3330347}

\bibitem{BernikDodson}
Vasilii Bernik and Maurice Dodson, \emph{Metric {D}iophantine approximation on
  manifolds}, Cambridge Tracts in Mathematics, vol. 137, Cambridge University
  Press, Cambridge, 1999.

\bibitem{Bugeaud5}
Yann Bugeaud, \emph{An inhomogeneous {J}arn\'\i k theorem}, J. Anal. Math.
  \textbf{92} (2004), 327--349. \MR{2072751}

\bibitem{Cassels}
J.~W.~S. Cassels, \emph{An introduction to {D}iophantine approximation},
  Cambridge Tracts in Mathematics and Mathematical Physics, No. 45, Cambridge
  University Press, New York, 1957.

\bibitem{EKL}
Manfred Einsiedler, Anatole Katok, and Elon Lindenstrauss, \emph{Invariant
  measures and the set of exceptions to {L}ittlewood's conjecture}, Ann. of
  Math. (2) \textbf{164} (2006), no.~2, 513--560. \MR{2247967}

\bibitem{Falconer_book}
Kenneth Falconer, \emph{Fractal geometry: {M}athematical foundations and
  applications}, John Wiley \& Sons, Ltd., Chichester, 1990.

\bibitem{Gallagher2}
Patrick Gallagher, \emph{Metric simultaneous diophantine approximation}, J.
  London Math. Soc. \textbf{37} (1962), 387--390. \MR{0157939}

\end{thebibliography}
%

\providecommand{\bysame}{\leavevmode\hbox to3em{\hrulefill}\thinspace}
\providecommand{\MR}{\relax\ifhmode\unskip\space\fi MR }
\providecommand{\MRhref}[2]{%
  \href{http://www.ams.org/mathscinet-getitem?mr=#1}{#2}
}
\providecommand{\href}[2]{#2}

\end{document}